\documentclass[12pt, a4paper,oneside]{amsart}

\usepackage[a4paper,inner=2.5cm,outer=2cm,top=3cm,bottom=3cm,pdftex]{geometry}
\usepackage[utf8]{inputenc}
\usepackage[T2A]{fontenc}
\usepackage{amssymb,amsmath}
\usepackage[american]{babel}
\usepackage{amsopn}
\usepackage{enumerate}
\usepackage{mathrsfs}
\usepackage{wasysym}
\usepackage{color}
\usepackage{fancyhdr}
\newtheorem{theorem}{Theorem}[section]
\newtheorem{lemma}[theorem]{Lemma}
\newtheorem{definition}[theorem]{Definition}

\newtheorem{proposition}[theorem]{Proposition}
\newtheorem{remark}[theorem]{Remark}

\newcommand\RR{{{\mathbb R}}}

\newcommand{\eps}{\varepsilon}

\newcommand{\reff}[1]{(\ref{#1})}

\newcommand{\p}{\partial}
\newcommand{\La}{\Delta}

\newcommand{\beq}{\begin{eqnarray*}}
\newcommand{\enq}{\end{eqnarray*}}
\newcommand{\beqq}{\begin{eqnarray}}
\newcommand{\enqq}{\end{eqnarray}}
\newcommand{\ben}{\begin{equation}\label}
\newcommand{\enn}{\end{equation}}
\newcommand{\bef}{\begin{proof}}
\newcommand{\enf}{\end{proof}}
\begin{document}
\title
{Blow-up criterion and examples of global solutions of forced Navier-Stokes equations}

\author[ Di Wu]
{ Di Wu}
\date{}
\address{\noindent \textsc{Di Wu, Institute de Math\'ematiques de Jussieu-Paris Rive Gouche UMR CNRS 7586, Univerist\'e Paris Diderot, 75205, Paris, France}}
\email{di.wu@imj-prg.fr}


\begin{abstract}
In this paper we first show a blow-up criterion for solutions to the Navier-Stokes equations with a time-independent force by using the profile decomposition method. Based on the orthogonal properties related to the profiles, we give some examples of global solutions to the Navier-Stokes equations with a time-independent force, whose initial data are large.
 
\end{abstract}

\maketitle

\section{Introduction}
We consider the incompressible Navier-Stokes equations with a time independent external force in $\RR^3$,
\beq
(NSf)~~\left\{
  \begin{array}{ll}
   \p_t u_f-\La u_f+u_f\cdot\nabla u_f=f-\nabla p,\\
   \nabla\cdot u=0,\\
   u_f|_{t=0}=u_0 
  \end{array}
\right.
\enq
for $(t,x)\in(0,T)\times\RR^3$, where $u_f$ is the velocity vector field, $f(x)$ is the given external force defined in $\RR^3$ and $p(t,x)$ is the associated pressure function. In this paper, we study the blow-up criterion for $(NSf)$.

\subsection{Blow-up problem in critical spaces}
To put our results in perspective, we first recall the Navier-Stokes equations (without external force) blow-up problem in critical spaces.  Consider 
 the Navier-Stokes system: 
\beq
(NS)~~\left\{
  \begin{array}{ll}
   \p_t u-\La u+u\cdot\nabla u=-\nabla \pi,\\
   \nabla\cdot u=0,\\
   u_f|_{t=0}=u_0 
  \end{array}
\right.
\enq
where $u(t,\cdot):\RR^3\to \RR^3$ is the unknown velocity field.\\
The spaces $X$ appearing in the chain of continuous embeddings
\beq
\dot{H}^{\frac{1}{2}}\hookrightarrow L^3\hookrightarrow\dot{B}^{-1+\frac{3}{p}}_{p,q}\hookrightarrow\dot{B}^{-1+\frac{3}{p'}}_{p',q'}, ~(3<p\leq p'<\infty,3<q\leq q'<\infty)
\enq 
are all critical with respect to the Navier-Stokes scaling in that $\|u_{0,\lambda}\|_X\equiv\|u\|_X$ for all $\lambda>0$, where $u_{0,\lambda}:=\lambda u(\lambda x)$ is the initial data which evolves as $u_{\lambda}:=\lambda u(\lambda^2 t,\lambda x)$, as long as $u_0$ is the initial data for the solution $u(t,x)$. While the larger spaces $\dot{B}^{-1+\frac{3}{p}}_{p,\infty}$, $\mathrm{BMO}^{-1}$ and $\dot{B}^{-1}_{\infty,\infty}$ are also critical spaces and global well-posedness is known for the first two for small enough initial data in those spaces thanks to \cite{CM2,KT2,FP2} (but only for finite $p$ in the Besov case, see \cite{BP2}), the ones in the chain above guarantee the existence and uniqueness of local-in-time solutions for any initial data. Specifically, there exist corresponding spaces $X_T=X_T((0,T)\times\RR^3)$ such that for any $u_0\in X$, there exists $T>0$ and a unique strong solution $u\in X_T$ to the corresponding Duhamel-type integral equation,
\beqq\label{Duhamel}
\begin{split}
u(t)&=e^{t\La }u_0-\int_0^te^{(t-s)\La }\mathbb{P}\nabla\cdot(u(s)\otimes u(s))ds\\
&=e^{t\La }u_0+B(u,u),
\end{split}
\enqq 
where
\beq
&(v\otimes w)_{j,k}:=v_j w_k,~~[\nabla\cdot(v\otimes w)]:=\sum_{k=1}^3\p_k(v_j w_k)~\\
&\mathrm{and}~\mathbb{P}v:=v+\nabla(-\La)^{-1}(\nabla\cdot v),
\enq 
which results from applying the projection onto divergence-free vector fields operator $\mathbb{P}$ on $(NS)$ and solving the resulting nonlinear heat equation. Moreover, $X_T$ is such that any $u\in X_T$ satisfying $(NS)$ belongs to  $ C([0,T],X)$. Setting
\beq
T_{X_T}^*(u_0):=\sup\{T>0|\exists !u:=NS(u_0)\in X_T~\mathrm{solving}~(NS)\}
\enq 
the Navier-Stokes blow-up problem is:\\
\textbf{Question:}
\beq
\mathrm{Does}\sup_{0<t<T^*_{X_T}(u_0)}\|u(t,\cdot)\|_X<\infty~~\mathrm{imply}~~\mathrm{that}~~ T^*_{X_T}(u_0)=\infty?
\enq 
In the important work \cite{ESS12} of Escauriaza-Seregin-Šverák, it was established that for $X=L^3(\RR^3)$, the answer is yes. This extended a result in the foundational work of Leray \cite{L2} regarding the blow-up of $L^p(\RR^3)$ norms at a singularity with $p>3$, and of the ``Ladyzhenskaya-Prodi-Serrin'' type mixed norms $L^s_t(L^p_x),~\frac{2}{s}+\frac{3}{p}=1,~p>3$, establishing a difficult ``end-point'' case of those results. In \cite{gkp12}, based on the work \cite{KK2}, I. Gallagher, G. S. Koch, F. Planchon gave an alternative proof this result in the setting of strong solutions using the method of ``critical elements'' of C. Kenig and F. Merle. In \cite{gkp2}, I. Gallagher, G. S. Koch, F. Planchon  extended the method in \cite{gkp12} to give a positive answer to the above question for $X=\dot{B}_{p,q}^{-1+\frac{3}{p}}(\RR^3)$ for all $3<p,q<\infty$ (see Definition \ref{besov}). Also in \cite{DA2}, D. Albritton proved a stronger blow-up criterion in $\dot{B}^{s_p}_{p,q}$ for $3<p,q<\infty$ and his proof is based on elementary splitting arguments and energy estimates. 

We recall the main steps of the method of ``critical elements'': assume the above question's answer is no for some $X$ and define 
\beq
\infty>A_c:=\inf\{\sup_{t\in[0,T^*_{X_T}(u_0))}\|NS(u_0)(t)\|_{X}\Big|u_0\in X~\mathrm{with}~T^*_{X_T}(u_0)<\infty\},
\enq  
where $NS(u_0)$ is a solution to $(NS)$ belonging to $C([0,T^*_{X_T}(u_0),X)$ with initial data $u_0\in X$.
And define the set of initial data generating ``critical elements''(possibly empty) as follows:
\beq
\mathcal{D}_c:=\{u_0\in X|T^*(u_0)<\infty,~\sup_{t\in[0,T^*(u_0))}\|NS(u_0)\|_{X}=A_c\}.
\enq 
The main steps are:
\begin{enumerate}
	\item If $A_c<\infty$, then $\mathcal{D}_c$ is non empty.
	\item If $A_c<\infty$, then any $u_0\in\mathcal{D}_c$ satisfies $NS(u_0)(t)\to0$ in $\mathcal{S}'$ as $t\nearrow T^*(u_0)$.
	\item If $A_c<\infty$, by backward uniqueness of the heat equation (see \cite{ESS22} ), for any $u_0\in\mathcal{D}_c$, there exists a $t_0\in(0,T^*(u_0))$ such that $NS(u_0)(t_0)=0$, which contradicts to the fact that $A_c<\infty$.
\end{enumerate}

In this paper, we consider the blow-up problem for the Navier-Stokes equation with a time-independent external force $f$, where $\La^{-1}f$ is small in $L^3$ and the initial data belongs to $L^3(\RR^3)$. 

According to Theorem \ref{well-posed-L3}, we know that there exists a universal constant $c>0$ such that, if the given external force satisfies  $\|\La^{-1}f\|_{L^3}<c$, then for any initial data $u_0\in L^3$, there exists a unique maximal time $T^*(u_0,f)>0$ and a unique  solution to $(NSf)$ $u_f$ belonging to $\mathcal{C}([0,T]; L^3((\RR^3))$ for any $T<T^*$ with initial data $u_0$. Again by Theorem \ref{well-posed-L3}, we have that if $T^*(u_0,f)=\infty$, then $u_f\in C([0,\infty),L^3(\RR^3))\cap L^{\infty}(\RR_+,L^3(\RR^3))$, and if $T^*(u_0)<\infty$, we have for any $p>3$ and $2<r<\frac{2p}{p-3}$,
\beq
\lim_{t\to T^*(u_0,f)}\|u_f-U_f\|_{\mathcal{L}^{r}([0,t],\dot{B}_{p,p}^{s_p+\frac{2}{r}})}=\infty,
\enq  
where $U_f\in L^3$ is the unique small steady-state solution to $(NSf)$ (for existence and uniqueness of small steady-state solution, see \cite{bbis2}) and the function space $\mathcal{L}^{r}_t(\dot{B}_{p,p}^{s_p+\frac{2}{r}})$ is defined in Definition \ref{time-besov}. However, the above criterion is on the corresponding perturbation solution instead of solution $u_f$.

In this paper, we give the following blow-up criterion for $(NSf)$: Let $\La^{-1}f$ be small in $L^3$, then
\beq
(BC)~~~\limsup_{0<t<T^*(u_0,f)}\|u_f(t,\cdot)\|_{L^3}<\infty\Rightarrow T^*(u_0,f)=\infty.
\enq  
We use a profile decomposition for the solutions to $(NSf)$ to prove the above result. Precisely, the decomposition enables us to construct a connection between the forced  and the unforced equation, which provides the blow-up information from the unforced solution to the forced solution. More precisely, we can decompose $u_f$  in a form  consisting of the sum of profiles of solutions to $(NS)$, a solution to $(NSf)$ and a remainder. We show that the blow-up information of $u_f$ is determined by the blow-up information of the profiles of solutions to $(NS)$ by an argument using the scaling  property of those solutions.  
    Compared with the ``critical element'' roadmap, we avoid using backward uniqueness of the heat equation (which is only true for the unforced case).  We also mention that the method used in \cite{DA2} can not be applied to our forced case, because the proof of \cite{DA2} relies on the following scaling property: if $u$ is solution to $(NS)$ with initial data $u_0$, then $\lambda u(\lambda^2 t,\lambda x )$  is also a solution to $(NS)$ with initial data $\lambda u_0(\lambda\cdot)$. However the above scaling property is not true for the Navier-Stokes equation with a time-independent force $f$ satisfying $\La ^{-1}f\in L^3$. In fact, for any solution $u_f$ to $(NSf)$ with initial data $u_0$, $\lambda u_f(\lambda^2 t,\lambda x )$ is  no longer a solution to $(NSf)$, unless $f$ is self-similar (which means $f(t,x)\equiv \lambda^3f(\lambda^2 t,\lambda x)$), hence does not satisfy $\La^{-1} f\in L^{3}$.
    (And his proof still relies on the backwards uniqueness of heat equation.)

We also point out that one can obtain a profile decomposition of solutions to the forced Navier-Stokes equation with an  external force $f\in\mathcal{L}^{r}(\RR_+,\dot{B}^{s_p+\frac{2}{r}-2}_{p,p})$ (Definition \ref{time-besov})  with $s_p+\frac{2}{r}>0$ and initial data bounded in $\dot{B}^{s_p}_{p,p}$ for any $3<p<\infty$ with a similar proof as in \cite{gkp12}. And by the same argument as the proof of Theorem \ref{Blow-up}, one can show the blow-up criterion as $(BC)$ by replacing $L^3$ by $\dot{B}^{s_p}_{p,p}$.

\subsection{Global Solutions to the Forced Navier-Stokes Equations}

The second topic of this paper is about the global solution to the incompressible Navier-Stokes equation with a small external force. 

As we mention before, $(NS)$ has a global solution if the initial data is small enough in the critical initial data space $L^3$, $\dot{B}^{s_p}_{p,\infty}$ or $\mathrm{BMO}^{-1}$. According to embedding $L^3\hookrightarrow\dot{B}^{s_p}_{p,\infty}(p<\infty)\hookrightarrow\mathrm{BMO}^{-1}\hookrightarrow\dot{B}^{-1}_{\infty,\infty}$ and the fact that $\dot{B}^{-1}_{\infty,\infty}$ is the critical initial data space for $(NS)$, we mention that all of these global wellposedness results require the initial data small enough in $\dot{B}^{-1}_{\infty,\infty}$.

Let us point out that none of the results mentioned are specific to $(NS)$, as they do not use the special structure of the nonlinear term in $(NS)$. In \cite{CG} J.-Y Chemin and I. Gallagher proved the global wellposedness result of $(NS)$ under a nonlinear smallness assumption on the initial, which may hold despite the fact that the data is large in $\dot{B}^{-1}_{\infty,\infty}$. 

Our purpose is different. Once an initial $u_0\in L^3$ generates a global solution to $(NS)$, we want  to construct a global solution to $(NSf)$ with a scaled enough initial data $\lambda^{-1}u_0(\lambda\cdot)$. This is done  by using the perturbation equation of $(NSf)$ and the orthogonal property of scales/cores. Hence we prove that for any initial data $u_0\in L^3$ (could be large in $\dot{B}^{-1}_{\infty,\infty}$)  generating a global solution to $(NS)$, after scaling, generates a global solution to $(NSf)$

The rest of this article is structured as follows. In Section 3, we give the proofs Theorem \ref{Blow-up} and Theorem \ref{examples}. Section 4 is devoted to showing the profile decomposition of solutions to $(NSf)$. In Section 5, a perturbation result for $(NS)$ is stated in an appropriate functional setting which provides the key estimate of Section 4. Finally in the Appendixs, we recall some well-posedness results for $(NSf)$ and the corresponding steady-state equation. Also we collect  standard Besov space estimates used throughout the paper in it.

\section{Notation and Statement of the Result}
Let us first recall the definition of Besov spaces, in dimension $d\geq1$.
\begin{definition}\label{besov}
	Let $\phi$ be a function in $\mathcal{S}(\RR^d)$ such that $\hat{\phi}=1$ for $|\xi|\leq 1$ and $\hat{\phi}=0$ for $|\xi|>2$, and define $\phi_j:=2^{dj}(2^{j}x)$. Then the frequency localization operators are defined by 
	\beq
	S_j:=\phi_j*\cdot,~~\La _j:=S_{j+1}-S_j.
	\enq 
	Let $f$ be in $\mathcal{S}'(\RR^d)$. We say $f$ belongs to $\dot{B}^{s}_{p,q}$ if 
	\begin{enumerate}
      \item the partial sum $\sum_{j=-m}^{m}\La_j f$ converges to $f$ as a tempered distribution if $s<\frac{d}{p}$ and after taking the quotient with polynomials if not, and
      \item 
      \beq
      \|f\|_{\dot{B}^{s}_{p,q}}:=\|2^{js}\|\La_j f\|_{L^p_x}\|_{\ell^q_j}<\infty.
      \enq 
    \end{enumerate}
\end{definition}

We refer to \cite{CL2} for the introduction of the following type of space in the context of the Navier-Stokes equations. 
\begin{definition}\label{time-besov}
	Let $u(\cdot,t)\in\dot{B}^{s}_{p,q}$ for a.e. $t\in (t_1,t_2)$ and let $\La_j$ be a frequency localization with respect to the $x$ variable (see Definition \ref{besov}). We shall say that $u$ belongs to $\mathcal{L}^\rho([t_1,t_2],\dot{B}^{s}_{p,q})$ if 
	\beq
	\|u\|_{\mathcal{L}^\rho([t_1,t_2],\dot{B}^{s}_{p,q})}:=\|2^{js}\|\La_j u\|_{L^{\rho}([t_1,t_2]L^p_x})\|_{\ell^q_j}<\infty.
	\enq 
\end{definition}
Note that for $1\leq\rho_1\leq q\leq\rho_2\leq\infty$, we have
\beq
L^{\rho_1}([t_1,t_2],\dot{B}^{s}_{p,q})\hookrightarrow\mathcal{L}^{\rho_1}([t_1,t_2],\dot{B}^{s}_{p,q})\hookrightarrow\mathcal{L}^{\rho_2}([t_1,t_2],\dot{B}^{s}_{p,q})\hookrightarrow L^{\rho_2}([t_1,t_2],\dot{B}^{s}_{p,q}).
\enq
Let us introduce the following notations (introduced in \cite{gkp2}): we define $s_p:=-1+\frac{3}{p}$ and 
\beq
\mathbb{L}^{a:b}_{p}(t_1,t_2):=\mathcal{L}^{a}([t_1,t_2];\dot{B}^{s_p+\frac{2}{a}}_{p,p})\cap\mathcal{L}^{b}([t_1,t_2];\dot{B}^{s_p+\frac{2}{b}}_{p,p}),
\enq 
\beqq
\mathbb{L}^a_{p}:=\mathbb{L}^{a:a}_{p},~\mathbb{L}^{a:b}_{p}(T):=\mathbb{L}^{a:b}_{p}(0,T)~~
\mathrm{and}~~\mathbb{L}^{a:b}_{p}[T<T^*]:=\cap_{T<T^*}\mathbb{L}^{a:b}_{p}(T).
\enqq 
\begin{remark}\label{notation}
	We point out that according to our notations, $u\in\mathbb{L}^{a:b}_{p,q}[T<T^*]$ merely means that $u\in\mathbb{L}^{a:b}_{p,q}(T)$ for any $T<T^*$ and does not imply that $u\in\mathbb{L}^{a:b}_{p,q}(T^*)$(the notation does not imply any uniform control as $T\nearrow T^*$).
\end{remark}
Now let us state our main result.
\begin{theorem}[Blow-up Criterion]\label{Blow-up}
    Suppose that $\|\La^{-1} f\|_{L^3}<c$, where $c$ is the small universal constant in Theorem \ref{well-posed-L3}. 
	Let $u_0\in L^3(\RR^3)$ be a divergence free vector field and $u_f=NSf(u_0)\in C([0,T^*(u_0,f)),L^{3}(\RR^3))$, where $T^*(u_0,f)$ is the maximal life span of $u_f$,  be the unique strong solution of $(NSf)$ with initial data $u_0$. If $T^*(u_0,f)<\infty$, then 
	\beqq\label{blow-up}
	\limsup_{t\to T^*(u_0,f)}\|u(t)\|_{L^3(\RR^3)}=\infty.
	\enqq 
\end{theorem}
\begin{remark}
	
Our profile decomposition method is not only valid for a time-independent force, but also can be extended to more general time-dependent external force. For example: our method is valid for  solutions belonging to $C([0,T^*),L^3(\RR^3))$  constructed in \cite{CP2} with initial $u_0\in L^3$, where the external force $f$ can be written as $f=\nabla\cdot V$ and $\sup_{0<t<\infty}t^{1-\frac{3}{p}}\|V\|_{L^{\frac{p}{2}}}$ is small enough for some $3<p\leq 6$. 
Actually our method only depends on the smallness of $U_f$ and the continuity in time of solutions in space $L^3$, which are similar ($U_f$ can by replaced by some small solution with small initial data in $L^3$ constructed in \cite{CP2}) with  the solutions in \cite{CP2}, whose associated force is time-dependent. After that we can obtain $(BC)$ for any fixed small external force as above by a similar argument of the case that $f$ is time independent.
\end{remark}

Under the smallness assumption on the given external force, the following result is an example for the existence of global solution to $(NSf)$ whose initial data is large.

\begin{theorem}[Examples of Global Solutions]\label{examples}
Suppose that the external force $f$ is given and $\|\La^{-1}f\|_{L^3}<c_1$, where $c_1$ is a universal small enough positive constant in Proposition \ref{large-linear}.  Let $u_0\in L^3$ be a divergence free vector field and its corresponding solution $u$ to $(NS)$ belongs to $ C([0,\infty),L^3(\RR^3))$. Then there exist $\lambda>0$ depending on $u_0$ and $f$ such that
\beqq
u_{f}:=NSf(u_{0,\lambda})\in C([0,\infty),L^3(\RR^3)),
\enqq  
where $u_{0,\lambda}(\cdot)=\lambda^{-1}u_0(\lambda^{-1}\cdot)$.
\end{theorem}

\section{Proof of the two main results}
\subsection{The blow-up criterion}
Suppose that $\|\La^{-1} f\|_{L^3}<c$ is a fixed external force.

Let us define 
\beqq
\begin{split}
A_c:=\sup\{A>0|&\sup_{t\in[0,T^*(u_0,f))}\|NSf(u_0)(t)\|_{L^3}\leq A\\
&\Longrightarrow T^*(u_0,f)=\infty, \forall u_0\in L^3(\RR^3)\}.
\end{split}
\enqq 
Note that $A_c$ is well-defined by small-data results. 
If $A_c$ is finite, then $A_c$ can be rewritten as
\beq
A_c=\inf\{\sup_{t\in[0,T^*(u_0,f))}\|NSf(u_0)(t)\|_{L^3}|u_0\in L^3~\mathrm{with}~T^*(u_0,f)<\infty\}.
\enq 
In the case when $A_c<\infty$, we introduce the (possibly empty) set of initial data generating a critical element as follows:
\beq
\mathcal{D}_c:=\{u_0\in L^3(\RR^3)|T^*(u_0,f)<\infty,~\sup_{t\in[0,T^*(u_0,f))}\|NSf(u_0)(t)\|_{L^3}=A_c\}.
\enq
Before proving Theorem \ref{Blow-up}, we prove the above set is empty.
\begin{proposition}[$\mathcal{D}_c$ is empty]\label{critical-elements}
	Suppose that $A_c<\infty$, then  $\mathcal{D}_c=\emptyset$.
\end{proposition}
\begin{proof}
    We prove the proposition by contradiction.
	Assume $\mathcal{D}_c\neq\emptyset$, we take a $u_{0,c}\in\mathcal{D}_c$ and denote $u_c=NSf(u_{0,c})$. By the definition of $\mathcal{D}_c$, we have $T^*(u_{0,c},f)<\infty$ and 
	\beq
	\sup_{t\in[0,T^*(u_{0,c},f))}\|NSf(u_{0,c})(t)\|_{L^3}=A_c.
	\enq 
	We choose a sequence $(s_n)_{n\in\mathbb{N}}\subset[0,T^*(u_{0,c},f))$ such that $s_n\nearrow T^*(u_{0,c},f)$. Let $u_{0,n}:=u_c(s_n)$ and $u_n:=NSf(u_{0,n})$. 
	Since $A_c<\infty$, we know that $(u_{0,n})_{n\in\mathbb{N}}$ is a bounded sequence in $L^3(\RR^3)$ and 
	 \beq
	\sup_{t\in[0,T^*(u_{0,n},f))}\|u_n(t)\|_{L^3}=A_c.
	\enq 
	By Theorem \ref{profile-NSf} with the same notation, for any $t\leq\tau_n$, $u_n$ has the following profile decomposition , for any $J\geq J_0$ and $n\geq n(J_0)$,
	\beq
	u_n=U^1+\sum_{j=2}^J\Lambda_{j,n}U^j+w^J_n+r^J_n,
	\enq 
	where $\tau_n=\min_{j\in I}\{\lambda_{j,n}^2T^j\}$. After reordering, we can write 
	\beq
	u_n=\sum_{j=1}^J\Lambda_{j,n}U^j+w^J_n+r^J_n
	\enq 
	with $\Lambda_{j_0,n}\equiv\mathrm{Id}$ for some $1\leq j_0\leq J_0$ and for $j\leq J$ and $n$ large enough,
	\beq
	\forall j\leq k\leq J_0,~~\lambda_{j,n}^2T^*_j\leq \lambda_{k,n}^2 T^*_k.
	\enq 
	First we claim that $j_0>1$. In fact,
	by Theorem \ref{profile-NSf},
	\beq
	\lambda_{1,n}T^*_1\leq T^*(u_{0,n},f)=T^*(u_{0,c},f)-s_n\to0,~~\mathrm{as}~~n\to\infty,
	\enq
	which implies that 
	\beq
	\lim_{n\to\infty}\lambda_{1,n}=0.
	\enq 
	Hence $j_0>1$, which implies that with the new ordering $U^1=NS(\phi_1)$, and $T_1^*<\infty$.\\
	
	Now we take $s\in(0,T^*_1)$ and let $t_n=\lambda_{1,n}^2s$. According to Proposition \ref{orth-L3}, we have 
	\beq
	A_c^3\geq \|u_n(t_n)\|^3_{L^3}
	\geq \|U^1(s)\|_{L^3}^3+\eps(n,s),
	\enq 
	where $\lim_{n\to\infty}\eps(n,s)=0$ for any fixed $s$. By the blow-up criterion for the Navier-Stokes equation (see \cite{gkp12})
	\beq
	\limsup_{t\to T^*_1}\|U^1(t)\|_{L^3(\RR^3)}=\infty,
	\enq 
	then we choose a $s_0\in(0,T^*_1)$ such that 
	\beq
	\|U^1(s_0)\|_{L^3(\RR^3)}>2A_c.
	\enq 
	And we can take a corresponding $n_0:=n(s_0)$ such that $|\eps(n_0,s_0)|\leq A_c^3$. Then we get
	\beq
	A^3_c>8A^3_c-A^3_c=7A^3_c
	\enq 
	which contradicts the fact that $A_c<\infty$. Then $\mathcal{D}_c=\emptyset$.  
	\end{proof}
Now we prove Theorem \ref{Blow-up} by contradiction.
\begin{proof}[Proof of Theorem \ref{Blow-up}]
    We suppose that $A_c<\infty$ which means \reff{blow-up} fails.\\
    Let us consider a sequence $u_{0,n}$ bounded in the space $L^3$ such that the life span of $NSf(u_{0,n})$ satisfies $T^*(u_{0,n},f)<\infty$ for each $n\in\mathbb{N}$ and such that $$A_n:=\sup_{t\in[0,T^*(u_{0,n},f))}\|NSf(u_{0,n})\|_{L^3(\RR^3)}$$ satisfies
	\beq
	A_c\leq A_n ~~\mathrm{and}~~A_n\to A_c,~~n\to\infty.
	\enq 
	Then by Theorem \ref{profile-NSf} and after reordering as above, we have for any $J\geq J_0$ and $n\geq n(J_0)$
	\beq
	u_n:=NSf(u_{0,n},f)=\sum_{j=1}^{J}\Lambda_{j,n}U^j+w^J_n+r^J_n,\forall t\in[0,\tau_n]
	\enq 
	and for any $n\geq n_0(J_0)$, recalling that $T^*_j$ is the life span of $U^j$
	\beq
	\forall j\leq k\leq J_0,~~\lambda_{j,n}^2T^*_j\leq \lambda_{k,n}^2 T^*_k,
	\enq 
	where $U^{j_0}=NSf(\phi_{j_0})$ ($j_0$ is such that $\Lambda_{j_0,n}\equiv1$) and $U^j=NS(\phi_j)$ for any $1\leq j\leq J_0$ with $j\neq j_0$. Theorem \ref{profile-NSf} also ensures that there $J_0$ such that $T^*_{J_0}<\infty$ (if not we would have $\tau_n\equiv\infty$ and hence $T^*(u_{0,n},f)\equiv\infty$, contrary to our assumption). 
	On the other hand, we recall that $U^{j_0}:=NSf(\phi_{j_0})$ with $1\leq j_0\leq J_0 $, where $\phi_{j_0}$ is a weak limit of $(u_{0,n})_{n\geq 1}$. Therefore by the above re-ordering,  two different cases need to be considered:
	\begin{itemize}
     \item $j_0=1$: the lower-bound of the life span of $u_n$ is controlled by the life span of $U^{j_0}=U^1=NSf(\phi_{j_0})$, which generates a critical element. 
     \item $j_0>1$: the lower-bound of the life span of $u_n$ is controlled by the life span of $\Lambda_{1,n}NS(\phi_1)$.
     \end{itemize}

	\textbf{Case 1: $j_0=1$}. In this case, by definition of $A_c$, we have $U^1=NSf(\phi_1)$, $\Lambda_{1,n}\equiv\mathrm{Id}$ and 
	\beqq\label{An}
	\sup_{s\in[0,T^*_1)}\|NSf(\phi_1)\|_{L^3}\geq A_c.
	\enqq  
	For any $s\in (0,T^*_1)$, setting $t_n:=\lambda_{1,n}^2s$, by Proposition \ref{orth-L3}
	\beq
	A_n^3&\geq& \sup_{t\in[0,T^*(u_{0,n},f))}\|NSf(u_{0,n})\|^3_{L^3}\geq \|NSf(u_{0,n})(t_n)\|^3_{L^3}\\
	&\geq& \|U^1(s)\|^3_{L^3}+\eps(n,s),
	\enq 
	where for any fixed $s\in[0,T^*_1)$
	\beq
	\lim_{n\to\infty}\eps(n,s)=0 .
	\enq 
	According to \reff{An} and the fact that $A_n\to A_c$ as $n\to\infty$, we infer that
	\beq
	\sup_{s\in[0,T^*_1)}\|NSf(\phi_1)\|_{L^3}=A_c,
	\enq 
	which means $\phi_1\in\mathcal{D}_c$. This fact contradicts Proposition \ref{critical-elements}. \\
	\textbf{Case 2 $j_0>1$:}. In this case, $U^1=NS(\phi_1)$ and $U^1$ satisfies that
	\beqq\label{t-T1}
	\limsup_{t\to T^*_1}\|U^1(t)\|_{L^3}=\infty,
	\enqq 
	and $\Lambda_{1,n}\neq\mathrm{Id}$.
	
    On the other hand for any $s\in (0,T^*_1)$, setting $t_n:=\lambda_{1,n}^2s$,
    \beq
	A_n^3&\geq& \sup_{t\in[0,T^*(u_{0,n},f))}\|NSf(u_{0,n})\|^3_{L^3}\geq \|NSf(u_{0,n})(t_n)\|^2_{\dot{H}^{1/2}}\\
	&\geq& \|U^1(s)\|^3_{L^3}+\eps(n,s),
	\enq 
    where 
	\beq
	\lim_{n\to\infty}\eps(n,s)=0,~~\forall s\in[0,T^*_1).
	\enq 
	Thanks to \reff{t-T1}, one can 
	take $s_0$ such that 
	\beq
	\|U^1(s_0)\|_{L^3}>2A_c
	\enq 
	and choose $n_0:=n(s_0)$ such that $\eps(n_0,s_0)\leq A^3_c$ and $A^3_{n_0}\leq 2A^3_c$, then we have 
	\beq
	2A^3_c&\geq& \|U^1(s_0)\|^3_{L^3}+\eps(n_0,s_0)\\
	&>&7A^3_c,
	\enq 
	which contradicts the fact that $A_c<\infty$. Then we prove that for any $u_0$, if $T^*(u_{0},f)<\infty$
	\beq
	\limsup_{t\to T^*(u_0,f)}\|NSf(u_0)\|_{L^3(\RR^3)}=\infty.
	\enq 
	Theorem \ref{Blow-up} is proved.

\end{proof}

\subsection{The global solutions to $(NSf)$}
In this part we focus on the existence of global solutions to $(NSf)$. In this paragraph we  assume that the given external force $f$ satisfies $\|\La^{-1}f\|_{L^3}<c_1$, where $c_1$ is the small constant given in Proposition \ref{large-linear}.

\begin{proof}[The proof of Theorem \ref{examples}]
Suppose that $f$ is the given external force and $\|\La^{-1}f\|_{L^3}<c_1$. According to Theorem \ref{well-posed-L3}, there exists a unique solution $U_f:=NSf(0)\in C(\RR_+,L^3)$ to $(NSf)$ with initial data $0$. Let $u_0\in L^3$ described in the theorem. Then its corresponding solution $u:=NS(u_0)$ to $(NS)$ belonging to $L^{\infty}(\RR_+,L^3(\RR^3))\cap \mathbb{L}^{1:\infty}_{p}(\infty)$, for any $3<p<\infty$.

We denote $u_\lambda=\Lambda_{\lambda}u$, where $\Lambda_{\lambda}u(t,x):=\lambda^{-1}u(\lambda^{-2}t,\lambda^{-1}x)$. It is easy to check that $\Lambda_\lambda u$ is a solution to $(NS)$ with the initial data $\lambda^{-1}u_0(\lambda\cdot)$ denoted by $u_{0,\lambda}$.

By Theorem \ref{well-posed-L3}, for any fixed $\lambda>0$ there exist a unique $T^*(u_{0,\lambda},f)>0$ and a unique solution $u_{f,\lambda}:=NSf(u_{0,\lambda})$ to $(NSf)$ such that $u_{f,\lambda}\in C([0,T^*(u_{0,\lambda},f)),L^3(\RR^3))$. We define $r_\lambda=u_{f,\lambda}-u_{\lambda}-U_f$, which is a solution to the following perturbation equation
\beq
\left\{
  \begin{array}{ll}
   \p_t r_\lambda-\La r_\lambda+\frac{1}{2}Q(r_\lambda,r_\lambda)+Q(r_\lambda,u_\lambda+U_f)=-Q(u_{\lambda},U_f),\\
   \nabla\cdot r_\lambda=0,\\
   r_\lambda|_{t=0}=0,
  \end{array}
\right.
\enq
where $Q(a,b)=\mathbb{P}(a\cdot\nabla b+ b\cdot\nabla a)$. Hence to prove that there exists a $\lambda_0>0$ such that $u_f\in C(\RR_+,L^3(\RR^3))$ it is enough to prove there exists a $\lambda_0$ such that  $r_{\lambda_0}\in\mathbb{L}^{p:\infty}_p(\infty)$ for some $3<p<5$.\\
In fact if $r_{\lambda_0}\in\mathbb{L}^{p:\infty}_p(\infty)$, we notice that $r_{\lambda_0}$ has the following integral form
\beq
r_{\lambda_0}(t)=B(r_{\lambda_0},_{\lambda_0}r)+2B(r_{\lambda_0},u_{\lambda_0}+U_f)+2B(u_{\lambda_0},U_f),
\enq
where $B$ is defined in \reff{Duhamel}. By the smooth effect of heat kernel and the product law of Besov space introduced in Proposition \ref{product-law}, we have that
\beq
\|B(r_{\lambda_0},r_{\lambda_0})\|_{\mathbb{L}^{\frac{p}{2}:\infty}_{\frac{p}{2}}(\infty)}+2\|B(r_{\lambda_0},u_{\lambda_0})\|_{\mathbb{L}^{\frac{p}{2}:\infty}_{\frac{p}{2}}(\infty)}\lesssim\|r_{\lambda_0}\|_{\mathbb{L}^{p:\infty}_{p}(\infty)}(\|r_{\lambda_0}\|_{\mathbb{L}^{p:\infty}_{p}(\infty)}+\|u_{\lambda_0}\|_{\mathbb{L}^{p:\infty}_{p}(\infty)}),
\enq
and there exists some $2<q<3$ such that
\beq
\|B(r_{\lambda_0},U_f)\|_{\mathbb{L}^{p:\infty}_{q}(\infty)}+\|B(u_{\lambda_0},U_f)\|_{\mathbb{L}^{p:\infty}_{q}(\infty)}\lesssim(\|r_{\lambda_0}\|_{\mathbb{L}^{p:\infty}_{p}(\infty)}+\|u_{\lambda_0}\|_{\mathbb{L}^{p:\infty}_{p}(\infty)})\|U_f\|_{L^3}.
\enq
According to the fact that for any $1\leq \bar{p}<3$, $\mathcal{L}^{\infty}(\RR_+,\dot{B}^{s_{\bar{p}}}_{\bar{p}}(\RR^3))\hookrightarrow L^{\infty}(\RR_+, L^{3}(\RR^3))$ and the smooth effect of heat kernel, we prove $r\in C(\RR_+,L^3(\RR^3))$.

Now we turn to prove that there exists a $\lambda_0>0$ such that $r_{\lambda_0}\in\mathbb{L}^{p:\infty}_p(\infty)$ with $3<p<5$.\\
According to Proposition \ref{orthogonal-product}, we have 
\beq
\lim_{\lambda\to0}\|Q(u_\lambda,U_f)\|_{\mathcal{L}^{p}(\RR_+,\dot{B}_{p,p}^{s_p+\frac{2}{p}-2})}\leq \lim_{\lambda\to0}\|u_\lambda\otimes U_f\|_{\mathcal{L}^{p}(\RR_+,\dot{B}_{p,p}^{s_p+\frac{2}{p}-1})}=0.
\enq
Hence there exists a $\lambda_0>0$ such that 
\beq
\|Q(u_{\lambda_0},U_{f})\|_{\mathcal{L}^{p}(\RR_+,\dot{B}_{p,p}^{s_p+\frac{p}{2}-2})}\leq \eps_0\mathrm{exp}\big(-C\|u_{\lambda_0}\|_{\mathcal{L}^{p}(\RR_+,\dot{B}_{p,p}^{s_p+\frac{2}{p}})}\big)
\enq 
provided that $\|u_{\lambda}\|_{\mathcal{L}^{p}(\RR_+,\dot{B}_{p,p}^{s_p+\frac{2}{p}})}=\|u\|_{\mathcal{L}^{p}(\RR_+,\dot{B}_{p,p}^{s_p+\frac{2}{p}})}$ is independent of $\lambda$. Here $\varepsilon_0$ and $C$ are constant in Proposition $\ref{large-linear}$. \\
Applying Proposition \ref{large-linear} we prove $r_{\lambda_0}\in\mathbb{L}^{p:\infty}_p(\infty)$ with $3<p<5$. Then we prove the theorem.

\end{proof}

\section{Profile decomposition}
In \cite{gkp12} a profile decomposition of solutions to the Navier-Stokes equations associated with data in $\dot{B}^{s_p}_{p,p}$ is proved for $d<p<2d+3$, thus extending the result of \cite{KK2}. In this section we use the idea of \cite{gkp12} to give a decomposition of solutions to the Navier-Stokes equations with a small external force and associated with initial data in $L^3$.

\subsection{Profile decomposition of bounded sequence in $L^3$}
Before stating the main result of this section, let us recall the following definition.
\begin{definition}\label{profile}
	We say that two sequences $(\lambda_{j,n},x_{j,n})_{n\in\mathbb{N}}\in ((0,\infty)\times\RR^3)^{\mathbb{N}}$ for $j\in\{1,2\}$ are orthogonal, and we write $(\lambda_{1,n},x_{1,n})_{n\in\mathbb{N}}\perp(\lambda_{2,n},x_{2,n})_{n\in\mathbb{N}}$, if
	\beqq\label{ortho-lambda}
	\lim_{n\to+\infty}\frac{\lambda_{1,n}}{\lambda_{2,n}}+\frac{\lambda_{2,n}}{\lambda_{1,n}}+\frac{|x_{1,n}-x_{2,n}|}{\lambda_{1,n}}=+\infty.
	\enqq 
	Similarly we say that a set of $(\lambda_{j,n},x_{j,n})_{n\in\mathbb{N}}$, for $j\in\mathbb{N},~j\geq1$, is orthogonal if for all $j\neq j'$, $(\lambda_{j,n},x_{j,n})_{n\in\mathbb{N}}\perp(\lambda_{j',n},x_{j',n})_{n\in\mathbb{N}}$.
\end{definition}
Next let us define, for any set of sequences $(\lambda_{j,n},x_{j,n})_{n\in\mathbb{N}}$ (for $j\geq1$), the scaling operator 
\beqq
\Lambda_{j,n}U_j(t,x):=\frac{1}{\lambda_{j,n}}U_j(\frac{t}{\lambda^2_{j,n}},\frac{x-x_{j,n}}{\lambda_{j,n}}).
\enqq
 
It is proved in \cite{K12} that any bounded (time-independent) sequence in $\dot{B}^{s_p}_{p,p}(\RR^3)$ may be decomposed into a sum of rescaled functions $\Lambda_{j,n}\phi_j$, where the set of sequences $(\lambda_{j,n},x_{j,n})_{n\in\mathbb{N}}$ is orthogonal, up to a small remainder term in $\dot{B}^{s_q}_{q,q}$, for any $q>p$. Since in this paper we only consider the initial data in $L^3$, we only state the  profile decomposition result of bounded sequences in $L^3$ in \cite{K12}. The precise statement is in the spirit of the pioneering work \cite{PG2}.
\begin{theorem}\label{profile-L3}
	Let $(\varphi_n)_{n\geq 1}$ be a bounded sequence of functions in $L^3(\RR^3)$ and let $\phi_1$ be any weak limit point of $(\varphi_n)_{n\in\mathbb{N}}$. Then, after possibly replacing $(\varphi_n)_{n\in\mathbb{N}}$ by a subsequence which we relabel $(\varphi_n)_{n\geq1}$, there exists a subsequence of profiles $(\phi_j)_{j\geq2}$ in $L^3(\RR^3)$, and a set of sequences $(\lambda_{j,n},x_{j,n})_{n\in\mathbb{N}}$	for $j\in\mathbb{N}$ with $(\lambda_{1,n},x_{1,n})\equiv(1,0)$ which are orthogonal in the sense of Definition \ref{profile} such that, for all $n, J\in\mathbb{N}$, if we define $\psi^J_n$ by
	\beq
	\varphi_n=\sum_{j=1}^J\Lambda_{j,n}\phi_j+\psi^J_n,
	\enq
	the following properties hold:
	\begin{itemize}
		\item the function $\psi^J_n$ is a remainder in the sense that for any $p>3$,
	          \beqq\label{remainder}
	          \lim_{J\to\infty}\big(\limsup_{n\to\infty}\|\psi^J_n\|_{\dot{B}^{s_p}_{p,p}}\big)=0;
	          \enqq 
	    \item There is a norm $\|\cdot\tilde{\|}_{L^3}$ which is equivalent to $\|\cdot\|_{L^3}$ such for each $n\in\mathbb{N}$,
	          \beq
	          \sum_{j=1}^{\infty}\|\phi_j\tilde{\|}^3_{L^3(\RR^3)}\leq \liminf_{n\to\infty}\|\varphi_n\tilde{\|}^3_{L^3(\RR^3)}
	          \enq 
	          and, for any interger $J$,
	          \beq
	          \|\psi^J_n\tilde{\|}_{L^3}\leq \|\varphi_n\tilde{\|}_{L^3}+o(1)
	          \enq
	          as $n$ goes to infinity.
	\end{itemize}
\end{theorem}
We mention that, in particular, for any $j\geq 2$, either $\lim_{n\to\infty}|x_{j,n}|=\infty$ or $\lim_{n\to\infty}\lambda_{j,n}\in\{0,\infty\}$ due to the orthogonality of scales/cores with $(\lambda_{1,n},x_{1,n})\equiv(1,0)$, and also that 
\beqq\label{sum-L3}
\sum_{j=1}^{\infty}\|\phi_j\|_{L^3}^3\lesssim\liminf_{n\to\infty}\|\varphi_n\|_{L^3}^3.
\enqq  
\subsection{Profile decomposition of solutions to $(NSf)$}

\begin{theorem}\label{profile-NSf}
  Suppose that $\|\La^{-1}f\|_{L^3}<c$, where $c$ is the small universal constant in Theorem \ref{well-posed-L3}.
  
  Let $(u_{0,n})_{n\in\mathbb{N}}$ be a bounded sequence of divergence-free vector fields in $L^3(\RR^3)$, and $\phi_1$ be any weak limit point of $\{u_{0,n}\}$. Then, after possibly relabeling the sequence due to the extraction of a subsequence following an application of Theorem \ref{profile-L3} with $\varphi_n:=u_{0,n}$, defining $u_n:=NSf(u_{0,n})$, $U^1:=NSf(\phi_1)\in C([0,T_1],L^3)$ and $U^j:=NS(\phi_j)\in C([0,T_j],L^3)$ for any $j\geq2$ (where $T_j$ is any real number smaller than $T^*_j$, where $T^*_j$ is the life span of $U^j$ for $j\geq1$, and $T^j=\infty$ if $T^*_j=\infty$), the following properties hold:
  \begin{itemize}
  	\item there is a finite (possibly empty) subset $I$ of $\mathbb{N}$ such that 
          \beq
          \forall j\in I,~~T^j<\infty~~\mathrm{and}~~ \forall j'\in\mathbb{N}\backslash I,~~U^{j'}\in  C(\RR_+,L^3(\RR^3)).
          \enq 
          Moreover setting $\tau_n:=\min_{j\in I}\lambda_{j,n}^2T^j$ if $I$ is nonempty and $\tau_n=\infty$ otherwise, we have
          \beqq\label{per-bounded}
          \sup_n\|u_n\|_{\mathcal{L}^{\infty}_{\tau_n}(L^3(\RR^3))}<\infty;
          \enqq
  	\item there exists some large $J_0\in\mathbb{N}$ such that for each $J>J_0$, there exists $N(J)\in\mathbb{N}$ such that for all $n>N(J)$, all $t\leq\tau_n$ and all $x\in\RR^3$, setting $w^J_n:=e^{t\La}(\psi^J_n)$ and defining $r^J_n$ by 
          \beqq
          u_n(t,x)=U^1+\sum_{j=2}^J\Lambda_{j,n}U^j+w^J_n+r^J_n,
          \enqq 
          then $w^J_n$ and $r^J_n$ are small remainders in the sense that, for any $3<p<5$,
          \beqq\label{wnrn}
          \lim_{J\to\infty}\big(\limsup_{n\to\infty}\|w^J_n\|_{\mathbb{L}_p^{1:\infty}(\infty)}\big)=\lim_{J\to\infty}\big(\limsup_{n\to\infty}\|r^J_n\|_{\mathbb{L}^{p:\infty}_p(\tau_n)}\big)=0.
          \enqq 
  \end{itemize}
\end{theorem}

We recall the following important orthogonality result without proof. Its proof is the same as the proof of Claim 3.3 of \cite{gkp12}, as it just depends on orthogonality property on scales/core.  To state the  result, note first that  an application of Theorem \ref{profile-NSf} yields a non- empty blow-up set $I\subset\{1,\dots, J_0\}$. Then we can re-order those first $J_0$ profiles, thanks to the orthogonality \reff{ortho-lambda} of the scales $\lambda_{j,n}$ so that for $n_0=n_0(J_0)$ sufficiently large, we have
\beqq\label{re-order}
\forall n\geq n_0,~~~1\leq j\leq j'\leq J_0\Rightarrow \lambda^2_{j,n}T^*_j\leq\lambda^2_{j',n}T^*_{j'}
\enqq  
(some of these terms may equal infinity).
\begin{proposition}\label{orth-L3}
	Let $(u_{0,n})_{n\geq 1}$ be a bounded sequence in $L^{3}$ and for which the set $I$ of blow-up profile indices resulting from an an application of Theorem \ref{profile-NSf} is non-empty. After re-ordering the profiles in the profile decomposition of $u_n:=NSf(u_{0,n})$ such that \reff{re-order} holds for some $J_0$, setting $t_n:=\lambda_{1,n}^2s$ for $s\in[0,T^*_1)$ one has (after possibly passing to a subsequence in $n$):
	\beqq
	\|u_n(t)\|_{L^3}^3=\|(\Lambda_{1,n}U^1)(t_n)\|^3_{L^3}+\|u(t_n)-(\Lambda_{1,n}U^1)(t_n)\|^3_{L^3}+\eps(n,s),
	\enqq 
	where $\eps(n,s)\to0$ as $n\to\infty$ for each fixed $s\in[0,T^*_1)$.
	
\end{proposition}
\begin{proof}[Proof of Theorem \ref{profile-NSf} ]
	Let $(u_{0,n})_{n\geq 1}$ be a bounded sequence in $L^{3}$. We first use Theorem \ref{profile-L3} to decompose the above sequence. \\
	Then with the notation of Theorem \ref{profile-NSf} 
	\beq
	u_{0,n}=\sum_{j=1}^J\Lambda_{j,n}\phi_j+\psi^J_n.
	\enq 
	 We define 
	\beq
	&u_n:=NSf(u_{0,n}),~~U^1:=NSf(\phi_1)\in C([0,T^*_1),L^3(\RR^3)),\\
	&U^j:=NS(\phi_j)\in C([0,T^*_j),L^3(\RR^3))~~\mathrm{and}~~w^J_n:=e^{t\La}(\psi^J_n).
	\enq 
	By \reff{remainder} and standard linear heat estimates we have 
	\beq
	\lim_{J\to\infty}\Big(\limsup_{n\to\infty}\|w^J_n\|_{\mathbb{L}^{1:\infty}_p(\infty)}\Big)=0.
	\enq 
	According to \reff{sum-L3}, we have  for any $p>3$
	\beqq\label{sum-besov}
	\Big\|\big(\|\phi_j\|_{\dot{B}^{s_p}_{p,p}}\big)_{j=1}^{\infty}\Big\|_{\ell^p}\lesssim \Big\|\big(\|\phi_j\|_{L^3}\big)_{j=1}^{\infty}\Big\|_{\ell^3}\lesssim\liminf_{n\to\infty}\|u_{0,n}\|_{L^3},
	\enqq
	which implies that, for any $j\geq 2$,
	\beq
	U^j\in\mathbb{L}^{1:\infty}_p(T<T^*_j)
	\enq 
	and there exists $J_0>0$ such that for any $j\geq J_0$, $T^*_j=\infty$. Moreover, for any $j\geq J_0$
	\beq
	U^j\in \mathbb{L}^{1:\infty}_p(\infty)~~\mathrm{and}~~\|U^j\|_{\mathbb{L}^{1:\infty}_p(\infty)}\lesssim\|\phi_j\|_{\dot{B}^{s_p}_{p,p}}.
	\enq 
	Hence, $I$ will be a subset of $\{1,\dots, J_0\}$ which proves the first part of the first statement in Theorem \ref{profile-NSf}.
	
	From now on, we restrict $p\in(3,5)$. By the local Cauchy theory we can solve $(NSf)$ with initial data $u_{0,n}$ for each integer $n$, and produce a unique  solution $u_n\in C([0,T^*_n),L^3(\RR^3))$, where $T^*_n$ is the life span of $u_n$. Now we define, for any $J\geq1$,
	\beq
	r^J_n:=u_n-\sum_{j=1}^{J}\Lambda_{j,n}U^j-w^J_n,
	\enq 
	where we recall that $\Lambda_{1,n}U^1=U^1$. We mention that the life span of $\Lambda_{j,n}U^j$ is $\lambda^2_{j,n}T^*_j$. Therefore, the function $r^J_n(t,x)$ is defined a priori for $t\in[0,t_n)$, where
	\beq
	t_n:=\min(T^*_n;\tau_n)
	\enq 
	with the notation of Theorem \ref{profile-NSf}. Our main goal is to prove that $r^J_n$ is actually defined on $[0,\tau_n]$ (at least if $J$ and $n$ are large enough), which will be a consequence of perturbation theory for the Navier-Stokes equations, see Proposition \ref{large-linear}. In  the process, we shall obtain the uniform limiting property, namely,
	\beqq
	\lim_{J\to\infty}\Big(\limsup_{n\to\infty}\|r^J_n\|_{\mathbb{L}^{p:\infty}_p(\tau_n)}\Big)=0.
	\enqq

	Let us write the equation satisfied by $r^J_n$. We adapt the same method as \cite{gkp2} and \cite{gkp12}.
	 It turns out to be easier to write that equation after a re-scaling in space-time. For convenience, let use re-order the functions $\Lambda_{j,n}U^j$, for $1\leq j\leq J_0$, in such a way that, for some $n_0=n_0(J_0)$ sufficiently large, we have as in \cite{gkp12},
	\beq
	\forall n\geq n_0,~~j\leq j'\leq J_0\Rightarrow \lambda_{j,n}T^*_j\leq \lambda_{j',n}T^*_{j'}.
	\enq 
	 And we define $1\leq j_0\leq J_0$ as the integer such that $\Lambda_{j_0,n}\equiv\mathrm{Id}$. And $\Lambda_{j_0,n}U^{j_0}=NSf(\phi_{j_0})=U_f+V^{j_0}$ (see Theorem \ref{well-posed-L3}).
	 We note that $\lambda^2_{j,n}T^*_j$ is the life span of $\Lambda_{j,n}U^j$.
	
	The inverse of our dilation/translation operator $\Lambda_{j,n}$ is 
	\beqq
	\Lambda_{j,n}^{-1}f(s,y):=\lambda_{j,n}f(\lambda^2_{j,n}s,\lambda_{j,n}y+x_{j,n}).
	\enqq 
	Then we define, for any integer $J$,
	\beq
	&j\leq J,~~U^{j,1}_n:= \Lambda^{-1}_{1,n}\Lambda_{j,n} U^j,~~R^{J,1}_n:=\Lambda^{-1}_{1,n} r^J_n,~~V^{j_0,1}_n=\Lambda_{1,n}^{-1}V^{j_0} \\
	& U^{1}_f:=\Lambda_{1,n}^{-1}U_f,~~  W^{J,1}_n:=\Lambda^{-1}_{1,n} w^J_n~~\mathrm{and}~~U^1_n:=\Lambda^{-1}_{1,n}u_n.
	\enq 
	Clearly we have
	\beq
	R^{J,1}=U^1_n-(\sum_{j=1}^J U^{j,1}_n+W^{J,1}_n)
	\enq 
	and $R^{J,1}_n$ is a divergence free vector field, solving the following system:
	\beqq
   \left\{
  \begin{array}{ll}
    \p_t R^{J,1}_n-\La R^{J,1}_n+\mathbb{P}(R^{J,1}_n\cdot\nabla R^{J,1}_n)+Q(R^{J,1}_n,U^{1}_f+G^{J,1}_n)=F^{J,1}_n,\\
   R^{J,1}_n|_{t=0}=0,
  \end{array}
  \right.
\enqq 
	where we recall that $\mathbb{P}:=\mathrm{Id}-\nabla\La^{-1}(\nabla\cdot)$ is the projection onto divergence free vector fields, and where
	\beq
	Q(a,b):=\mathbb{P}((a\cdot\nabla)b+(b\cdot\nabla)a)
	\enq 
	for two vector fields $a,b$. Finally we have defined
	\beq
	G^{J,1}_n:=\sum_{j\neq j_0}^JU^{j,1}_n+W^{J,1}_n+V_n^{j_0,1},
	\enq 
	and
	\beq
	F^{J,1}_n,=-\frac{1}{2}Q(W^{J,1}_n,W^{J,1}_n)-\frac{1}{2}\sum_{j\neq j'}^JQ(U^{j,1}_n,U^{j',1}_n)-\sum_{j=1}^JQ(U^{j,1}_n,W^{J,1}_n).
	\enq
	In order to use perturbative bounds on this system, we need a uniform control on the drift term $G^{J,1}_n$ and smallness of the forcing term $F^{J,1}_n$. The results are the following.
	\begin{lemma}\label{drift-term}
		Fix $T_1<T^*_1$. The sequence $(G^{J,1}_n)_{n\geq1}$ is bounded in $\mathcal{L}^p([0,T_1],\dot{B}_{p,p}^{s_p+\frac{2}{p}})$, uniformly in J, which means that 
		\beq
		\lim_{J\to\infty}\limsup_{n\to\infty}\|G^{J,1}_n\|_{\mathcal{L}^p([0,T_1],\dot{B}_{p,p}^{s_p+\frac{2}{p}})}=0.
		\enq 
	\end{lemma}
	 The proof of the above lemma is the same as the proof of Lemma 2.5 in \cite{gkp12}, as it just depends on orthogonality property on scales/core. 
	\begin{lemma}\label{source-term}
		Fix $T_1<T^*_1$. The source term $F^{J,1}_n$ goes to zero for each $J\in\mathbb{N}$, as $n$ goes to infinity, in the space $\mathcal{F}:= \mathcal{L}^p([0,T_1],\dot{B}_{p,p}^{s_p+\frac{2}{p}-2})+\mathcal{L}^{\frac{p}{2}}([0,T_1],\dot{B}_{p,p}^{s_p+\frac{4}{p}-2})$. In precisely,
		\beq
		\lim_{J\to\infty}\limsup_{n\to\infty}\|F^{J,1}_n\|_{\mathcal{F}}=0.
		\enq 
	\end{lemma}
	
	Assuming these lemmas to be true, the end of the proof of the theorem is a direct consequence of Proposition \ref{large-linear}.
	
\end{proof}

Now let us prove Lemma \ref{source-term}.
\begin{proof}[Proof of Lemma \ref{source-term}]
    We first notice that 
    \beq
    \begin{aligned}
	F^{J,1}_n,=&-\frac{1}{2}Q(W^{J,1}_n,W^{J,1}_n)-\frac{1}{2}\sum_{\begin{subarray}{c}
	j\neq j',j\neq j_0\\
	j'\neq j_0	
	\end{subarray}}^JQ(U^{j,1}_n,U^{j',1}_n)-\sum_{j=1,j\neq j_0}^JQ(U^{j,1}_n,W^{J,1}_n)\\
	&-\sum_{j=1,j\neq j_0}^JQ(U^{j_0,1}_n,U^{j,1}_n)-Q(U^{j_0,1}_n,W^{J,1}_n).
	\end{aligned}
	\enq 
	And we note that the structure of 
	\beq
	A_n^J:=-\frac{1}{2}Q(W^{J,1}_n,W^{J,1}_n)-\frac{1}{2}\sum_{\begin{subarray}{c}
	j\neq j',j\neq j_0\\
	j'\neq j_0	
	\end{subarray}}^JQ(U^{j,1}_n,U^{j',1}_n)-\sum_{j=1,j\neq j_0}^JQ(U^{j,1}_n,W^{J,1}_n)
	\enq 
	is the same as the $G^{J,0}_{n}$ of Lemma 2.7 in \cite{gkp12}. As a consequence of Lemma 2.7 in \cite{gkp12}, we obtain
	\beq
		\lim_{J\to\infty}\limsup_{n\to\infty}\|A^{J}_n\|_{\mathcal{F}}=0.
    \enq 
    Hence to finish the proof of Lemma \ref{source-term}, we need to show
    \beq
    \lim_{J\to\infty}\limsup_{n\to\infty}\|B^{J}_n\|_{\mathcal{F}}=0,
    \enq 
    where 
    \beq
    B^J_n:=-\sum_{j=1,j\neq j_0}^JQ(U^{j_0,1}_n,U^{j,1}_n)-Q(U^{j_0,1}_n,W^{J,1}_n).
    \enq 
	
	By product laws and scaling invariance, we first have
	\beq
	\|Q(U^{j_0,1}_n,W^{J,1}_n)\|_{\mathcal{L}^p([0,T_1],\dot{B}_{p,p}^{s_p+\frac{2}{p}-2})}
	&\lesssim&\|W^{J,1}_n\|_{\mathcal{L}^p([0,T_1],\dot{B}_{p,p}^{s_p+\frac{2}{p}})}\|U^{j_0,1}_n\|_{\mathcal{L}^{\infty}([0,T_0],\dot{B}^{s_p}_{p,p})}\\
	&\lesssim&\|\psi^J_n\|_{\dot{B}^{s_p}_{p,p}}\|\phi_{j_0}\|_{\dot{B}^{s_p}_{p,p}},
	\enq 
	implies that
	\beq
	\lim_{J\to\infty}\limsup_{n\to\infty}\|Q(U^{j_0,1}_n,W^{J,1}_n)\|_{\mathcal{L}^p([0,T_1],\dot{B}_{p,p}^{s_p+\frac{2}{p}-2})}=0.
	\enq 
	Now we are left with proving that 
	\beq
	\lim_{J\to\infty}\limsup_{n\to\infty}\|\sum_{j\neq j'}^JQ(U^{j,1}_n,U^{j',1}_n)\|_{\mathcal{F}}=0.
	\enq 
	We can write $\sum_{j=1,j\neq j_0}^JQ(U^{j_0,1}_n,U^{j,1}_n)$ as the following way:
	\beq
	\sum_{j=1,j\neq j_0}^JQ(U^{j_0,1}_n,U^{j,1}_n)=\sum_{j\neq j_0}Q(V^{j_0,1}_n,U^{j,1}_n)
	+\sum_{j\neq j_0}Q(\Lambda_{1,n}^{-1}U_f, U^{j,0}_n).
	\enq 
	Since for any $j\neq j_0$ $U^{j,1}_n\in\mathbb{L}^{1:\infty}_p(T_0)$, $V^{j_0,1}_n\in \mathbb{L}^{p:\infty}_{p}(T_0)$ and  $3<p<5$, by \reff{orth-a} in  Proposition \ref{orthogonal-product}, we have for all $j',j\neq j_0$
	\beq
	\lim_{n\to\infty}\|Q(U^{j,1}_n,V^{j_0,1}_n)\|_{\mathcal{L}^{\frac{p}{2}}([0,T_1],\dot{B}_{p,p}^{s_p+\frac{4}{p}-2})}=0.
	\enq 
	And according to $U_f\in L^{3}(\RR^3)$, we have
	\beq
	\lim_{n\to\infty}\|Q(U^{j,1}_n,\Lambda_{1,n}^{-1}U_f )\|_{\mathcal{L}^p([0,T_1],\dot{B}_{p,p}^{s_p+\frac{2}{p}-2})}=0
	\enq 
	by Proposition \ref{orthogonal-product}. By the above two relations, we have
	\beq
	\lim_{J\to\infty}\limsup_{n\to\infty}\|\sum_{j=1,j\neq j_0}^JQ(U^{j_0,1}_n,U^{j,1}_n)\|_{\mathcal{L}^p([0,T_1],\dot{B}_{p,p}^{s_p+\frac{2}{p}-2})}=0.
	\enq 
	Lemma \ref{source-term} is proved.
\end{proof}

\subsection{Orthogonality Property}
In this paragraph, we show the orthogonality properties used in the proof of Lemma \ref{source-term}. The first statement of Proposition \ref{orthogonal-product} is just a particular case of orthogonality property given in \cite{gkp2} (see the proof Lemma 3.3 in \cite{gip2}). By the same idea in \cite{gkp2}, we give a orthogonality property in the case that one of the element in the product is time-independent.
\begin{proposition}\label{orthogonal-product}
We assume that $(\lambda_{1,n},x_{1,n})_{n\in\mathbb{N}}$ and $(\lambda_{2,n},x_{2,n})_{n\in\mathbb{N}}$ are orthogonal. Let $T\in\RR_+\cup\{+\infty\}$. Then the following properties hold:
	\begin{enumerate}
		\item Let $p>3$ and $1-\frac{3}{p}<\frac{1}{a}<1$. Suppose that $v,w\in\mathcal{L}^{2a}([0,T],\dot{B}_{p,p}^{s_p+\frac{1}{a}})$.  Then we have
	          \beqq\label{orth-a}
	          \lim_{n\to\infty}\|(\Lambda_{1,n}v)(\Lambda_{2,n}w)\|_{\mathcal{L}^{a}([0,T_n],\dot{B}^{s_p+\frac{2}{a}-1}_{p,p})}=0,
	          \enqq 
	          where $T_n:=\min\{\lambda_{1,n}^2T,\lambda_{2,n}^2T\}$.
		\item Let $p>3$ and $2<r<\frac{2p}{p-3}$. Suppose that $U\in L^{3}(\RR^3)$ and $v\in\mathcal{L}^{r}([0,T],\dot{B}_{p,p}^{s_p+\frac{2}{r}})$.
	         \beqq\label{orth-r}
	         \lim_{n\to\infty}\|(\Lambda_{1,n}U)(\Lambda_{2,n}v)\|_{\mathcal{L}^{r}([0,T'_n],\dot{B}_{p,p}^{s_p+\frac{2}{r}-1})}=0,
	         \enqq
	         where $T'_n:=\lambda_{2.n}T$.
	\end{enumerate}
\end{proposition}
\begin{proof}
    As we mentioned above, \reff{orth-a} is   a particular case of orthogonality property given in \cite{gkp2}, we only need to prove the second statement of the proposition.
    
	 For any given $\eps>0$ one can find two compactly supported (in space and time) functions $v_{\eps }$ and $U_{\eps }$ such that 
	\beq
	\|v-v_{\eps}\|_{\mathcal{L}^{2a}([0,T],\dot{B}_{p,p}^{s_p+\frac{1}{a}})}+\|U-U_{\eps}\|_{L^3}\leq \eps .
	\enq 
	Product rules (along with the scale invariance of the scaling operators) gives that
	\beq
	&&\|(\Lambda_{1,n}v)(\Lambda_{2,n}(U-U_{\eps}))\|_{\mathcal{L}^{a}([0,T_n],\dot{B}^{s_p+\frac{2}{a}-1}_{p,p})}\\
	&&+\|(\Lambda_{1,n}(v-v_{\eps }))(\Lambda_{2,n}(U)\|_{\mathcal{L}^{a}([0,T_n],\dot{B}^{s_p+\frac{2}{a}-1}_{p,p})}\\
	&&+\|(\Lambda_{1,n}(v-v_{\eps }))(\Lambda_{2,n}(U-U_{\eps}))\|_{\mathcal{L}^{a}([0,T_n],\dot{B}^{s_p+\frac{2}{a}-1}_{p,p})}\lesssim\eps.
	\enq
	Then it is enough to prove that for  fixed $\eps>0$
	\beq
	\lim_{n\to\infty}\|(\Lambda_{1,n}U_{\eps })(\Lambda_{2,n}(v_{\eps})\|_{\mathcal{L}^{r}([0,T_n'],\dot{B}^{s_p+\frac{2}{r}-1}_{p,p})}=0.
	\enq 
	Again by Proposition \ref{product-law}, we have for some $3<q<\frac{3p}{p-3}$ and small enough $\delta>0$,
	\beq
	\|(\Lambda_{1,n}U_{\eps })(\Lambda_{2,n}(v_{\eps})\|_{\mathcal{L}^{r}([0,T_n'],\dot{B}^{s_p+\frac{2}{r}-1}_{p,p})}
	\lesssim\|\Lambda_{1,n}U_{\eps }\|_{\dot{B}_{q,q}^{s_q+\delta}}\|\Lambda_{2,n}w_{\eps }\|_{\mathcal{L}^{r}([0,T_n'],\dot{B}_{p,p}^{s_p+\frac{1}{a}-\delta})}.
	\enq 
	According to the fact that 
	\beq
	\|\Lambda_{1,n}U_{\eps }\|_{\dot{B}_{q,q}^{s_q+\delta}}\lesssim\lambda_{1,n}^{-\delta}\|U\|_{L^3},
	\enq 
	and
	\beq
	\|\Lambda_{2,n}v_{\eps }\|_{\mathcal{L}^{r}([0,T_n],\dot{B}_{p,p}^{s_p+\frac{2}{r}-\delta})}\lesssim\lambda_{2,n}^{\delta}\|v_{\eps }\|_{\mathcal{L}^{2a}([0,T],\dot{B}_{p,p}^{s_p+\frac{2}{r}})},
	\enq 
	we have
	\beq
	\|(\Lambda_{1,n}U_{\eps })(\Lambda_{2,n}(v_{\eps})\|_{\mathcal{L}^{r}([0,T_n'],\dot{B}^{s_p+\frac{2}{r}-1}_{p,p})}\lesssim \Big(\frac{\lambda_{2,n}}{\lambda_{1,n}}\Big)^{\delta}\to0,~~n\to\infty,
	\enq 
	if $\frac{\lambda_{2,n}}{\lambda_{1,n}}\to0$. Hence we prove \reff{orth-r}.
	
\end{proof}
\section{Estimates on perturbation equations}
Now we consider the following perturbation equation, 
\beqq
\left\{
  \begin{array}{ll}
    \p_t w-\La w+\frac{1}{2}Q(w,w)+Q(w,g) +Q(w,U)=f,\\
    w|_{t=0}=w_0,
  \end{array}
\right.
\enqq 
Let us state the following perturbation result.
\begin{proposition}\label{large-linear}
	Let $T\in\RR_+\cup\{+\infty\}$ and $3<p<5$. Suppose that $U\in L^3(\RR^3)$, $g\in\mathcal{L}^p([0,T],\dot{B}_{p,p}^{s_p+\frac{2}{p}}(\RR^3))$ and $f\in \mathcal{F}([0,T]):= \mathcal{L}^p([0,T],\dot{B}_{p,p}^{s_p+\frac{2}{p}-2})+\mathcal{L}^{\frac{p}{2}}([0,T],\dot{B}_{p,p}^{s_p+\frac{4}{p}-2})$. 
	
Suppose that for any $h\in \mathcal{L}^p([0,T],\dot{B}^{s_p+\frac{2}{p}}_{p,p})$, $$\|Q(h,U)\|_{\mathcal{L}^p([0,T],\dot{B}^{s_p+\frac{2}{p}-2}_{p,p})}<c_1\|h\|_{\mathcal{L}^p([0,T],\dot{B}^{s_p+\frac{2}{p}}_{p,p})},$$ where $c_1>0$ is a universal small constant. Then there exists a constant $C$ independent of $T$ and $\eps_0$ such that the following is true. If 
	\beq
	\|w_0\|_{\dot{B}^{s_p}_{p,p}}+\|f\|_{\mathcal{F}([0,T])}\leq \eps_0\mathrm{exp}\big(-C\|g\|_{\mathcal{L}^{p}([0,T],\dot{B}_{p,p}^{s_p+\frac{2}{p}})}\big),
	\enq  
	then $w\in \mathbb{L}^{p}_p(T)$ and 
	\beq
	\|w\|_{\mathbb{L}^p_p(T)}\leq C(\|w_0\|_{\dot{B}^{s_p}_{p,p}}+\|f\|_{\mathcal{F}([0,T])})\mathrm{exp}(C\|g\|_{\mathcal{L}^{p}([0,T],\dot{B}_{p,p}^{s_p+\frac{2}{p}})}).
	\enq 
\end{proposition}
The proof the proposition follows the estimates of \cite{gip2} (see in particular Proposition 4.1 and Theorem 3.1 of \cite{gip2}). The main difference is the absence of an exterior force and a small time-independent drift term in \cite{gip2}, but those terms are added with no difficulty to the estimates.
\begin{proof}
	By Proposition 4.1 of \cite{gip2}, for any $\alpha,\beta\in[0,T]$ , we have the following estimates
	\beqq\label{p-estimate}
	\begin{split}
	  \|w\|_{\mathbb{L}^{p:\infty}_p([\alpha,\beta])}&\leq K\|w(\alpha)\|_{\dot{B}^{s_p}_{p,p}}+K\|f\|_{\mathcal{F}([\alpha,\beta])}+K\|w\|^2_{\mathcal{L}^{p}([\alpha,\beta],\dot{B}_{p,p}^{s_p+2/p})}\\
	&+K(c_1+\|g\|_{\mathcal{L}^{p}([\alpha,\beta],\dot{B}_{p,p}^{s_p+2/p})})\|w\|_{\mathcal{L}^{p}([\alpha,\beta],\dot{B}_{p,p}^{s_p+2/p})}.
	\end{split}
	\enqq 
	We recall that $c_1$ is a small enough number such that 
	\beq
	Kc_1<\frac{1}{4}.
	\enq 
	And we claim that there exist $N$ real numbers $(T_i)_{1\leq i\leq N}$ such that $T_1=0$ and $T_N=T$, satisfying $[0,T]=\cup_{i=1}^{N-1}[T_i,T_{i+1}]$ and
	\beq
	\|g\|_{\mathcal{L}^p([T_i,T_{i+1}],\dot{B}_{p,p}^{s_p+\frac{2}{p}})}\leq \frac{1}{4K},~~\forall i\in\{i,1\dots,N-1\}
	\enq 
	Suppose that 
	\beqq\label{small}
	\|w_0\|_{\dot{B}^{s_p}_{p,p}}+\|f\|_{\mathcal{F}([0,T])}\leq \frac{1}{8KN(4K)^N}.
	\enqq 
	By time continuity we can define a maximal time $\tilde{T}\in \RR_+\cup\{\infty\}$ such that 
	\beq
	\|w\|_{\mathcal{L}^p([0,\tilde{T}],\dot{B}_{p,p}^{s_p+\frac{2}{p}})}\leq \frac{1}{4K}.
	\enq 
	If $\tilde{T}\geq T$ then the proposition is proved. Indeed, by \reff{p-estimate},  we have ,
	\beq
	\|w\|_{\mathcal{L}^p([T_i,T_{i+1}],\dot{B}_{p,p}^{s_p+\frac{2}{p}})}\leq K\|w(T_i)\|_{\dot{B}_{p,p}^{s_p}}+K\|f\|_{\mathcal{F}([T_i,T_{i+1}])}+\frac{3}{4}\|w\|_{\mathcal{L}^p([T_i,T_{i+1}],\dot{B}_{p,p}^{s_p+\frac{2}{p}})},
	\enq 
	which deduces that 
	\beq
	\|w\|_{\mathcal{L}^p([T_i,T_{i+1}],\dot{B}_{p,p}^{s_p+\frac{2}{p}})}\leq 4K(\|w(T_i)\|_{\dot{B}_{p,p}^{s_p}}+\|f\|_{\mathcal{F}([T_i,T_{i+1}])}).
	\enq 
	Hence according to \reff{p-estimate} ,
	\beq
	\|w\|_{\mathcal{L}^{\infty}([T_i,T_{i+1}],\dot{B}_{p,p}^{s_p})}&\leq& K\|w(T_i)\|_{\dot{B}_{p,p}^{s_p}}+K\|f\|_{\mathcal{F}([T_i,T_{i+1}])}+\frac{3}{4}\|w\|_{\mathcal{L}^p([T_i,T_{i+1}],\dot{B}_{p,p}^{s_p+\frac{2}{p}})}\\
	&\leq&4K(\|w(T_i)\|_{\dot{B}_{p,p}^{s_p}}+\|f\|_{\mathcal{F}([T_i,T_{i+1}])}).
	\enq
	Therefore,
	\beq
      \|w(T_{i})\|_{\dot{B}^{s_p}_{p,p}}\leq (4K)^{i-1}(\|w(0)\|_{\dot{B}_{p,p}^{s_p}}+\|f\|_{\mathcal{F}([0,T])}),
	\enq 
	which implies that 
	\beq
	\|w\|_{\mathcal{L}^p([T_i,T_{i+1}],\dot{B}_{p,p}^{s_p+\frac{2}{p}})}&\leq (4K)^{i}(\|w(0)\|_{\dot{B}_{p,p}^{s_p}}+\|f\|_{\mathcal{F}([0,T])}).
	\enq
	 Hence,
	\beq
	\|w\|_{\mathcal{L}^p([0,T],\dot{B}_{p,p}^{s_p+\frac{2}{p}})}\leq N(4K)^N(\|w(0)\|_{\dot{B}_{p,p}^{s_p}}+\|f\|_{\mathcal{F}([0,T])}).
	\enq 
	Take $N\sim \|g\|_{\mathcal{L}^p([0,T],\dot{B}_{p,p}^{s_p+\frac{2}{p}})}$, we have
	\beq 
	\|w\|_{\mathcal{L}^p([0,T],\dot{B}_{p,p}^{s_p+\frac{2}{p}})}\lesssim (\|w(0)\|_{\dot{B}_{p,p}^{s_p}}+\|f\|_{\mathcal{F}([0,T])})\mathrm{exp}(C\|g\|_{\mathcal{L}^p([0,T],\dot{B}_{p,p}^{s_p+\frac{2}{p}})}).
	\enq 
	And by \reff{p-estimate}, we have
	\beq
	\|w\|_{\mathcal{L}^{\infty}([0,T],\dot{B}^{s_p}_{p,p})}\lesssim (\|w_0\|_{\dot{B}^{s_p}_{p,p}}+\|f\|_{\mathcal{F}([0,T])})\mathrm{exp}(C\|g\|_{\mathcal{L}^p([0,T],\dot{B}_{p,p}^{s_p+\frac{2}{p}})}).
	\enq 
	Thus the proposition is proved in the case $\tilde{T}\geq T$.
	
	Now we turn to the proof in the case of $\tilde{T}<T$. We define an integer $K\in\{1,\dots,N-1\}$ such that 
	\beq
	T_k\leq \tilde{T}<T_{k+1}.
	\enq 
	Then for any $i\leq k-1$, we have
	\beq
	\|w\|_{\mathcal{L}^p([T_i,T_{i+1}],\dot{B}_{p,p}^{s_p+\frac{2}{p}})}&\leq (4K)^{i}(\|w(0)\|_{\dot{B}_{p,p}^{s_p}}+\|f\|_{\mathcal{F}([0,T])}),
	\enq
	and 
	\beq
      \|w(T_{i})\|_{\dot{B}^{s_p}_{p,p}}\leq (4K)^{i-1}(\|w(0)\|_{\dot{B}_{p,p}^{s_p}}+\|f\|_{\mathcal{F}([0,T])}).
	\enq 

    The same arguments as above also apply on the interval $[T_k,T]$ and yield,
    \beq
    \|w\|_{\mathcal{L}^{p}([T_k,T],\dot{B}_{p,p}^{s_p+\frac{2}{p}})}\leq (4K)^{N}\|w_0\|_{\dot{B}^{s_p}_{p,p}}+CNK^2\|f\|_{\mathcal{F}([0,T])},
    \enq 
    and 
    \beq
    \|w\|_{\mathcal{L}^{\infty}([T_k,T],\dot{B}^{s_p}_{p,p})}\leq (4K)^{N}\|w_0\|_{\dot{B}^{s_p}_{p,p}}+CNK^2\|f\|_{\mathcal{F}([0,T])}.
    \enq 
    Therefore we have 
    \beq
    \|w\|_{\mathcal{L}^p([0,T],\dot{B}^{s_p+\frac{2}{p}}_{p,p})}\leq N(4K)^N(\|w_0\|_{\dot{B}^{s_p}_{p,p}}+\|f\|_{\mathcal{F}([0,T])})<\frac{1}{4k},
    \enq 
    which contradicts to the maximality of $\tilde{T}$.

	\end{proof}

\section{Appendix}
\subsection{Some results on the steady-state Navier-Stokes equations}
In this part, we  recall some existence results on the steady state Navier-Stokes equations, and the Navier-Stokes equations equipped with the same time-independent external force. The steady state Navier-Stokes system is defined as follows,
\beq
(SNS)\left\{
  \begin{array}{ll}
   -\La U+U\cdot\nabla=f-\nabla\Pi,\\
   \nabla\cdot U=0,
  \end{array}
\right.
\enq 
where $f(x)$ is the external force defined on $\RR^3$.  Since we only care about the case of $U\in L^3$, we state the following result for $\La^{-1}f\in L^3$ without proof, which is a consequence of Theorem 2.2 in \cite{bbis2}.
\begin{proposition}\label{steady}
	There exists an absolute constant $\delta>0$ with the following property. If $f\in\mathcal{S}'$ satisfies $\La^{-1}f\in L^3(\RR^3)$ and
	\beq
	\|\La^{-1}f\|_{L^3(\RR^3)}<\delta,
	\enq
	then there exists a unique solution to (SNS) such that 
	\beq
	\|U\|_{L^3}\leq 2\|\La^{-1}f\|_{L^3}<2\delta.
	\enq
\end{proposition}

The following wellposedness result of the incompressible Navier-Stokes equation with a time-independent force is a special case of the wellposedness result of Theorem 2.7 in \cite{DW2}. Hence we recall it without proof.

\begin{theorem}\label{well-posed-L3}
	Suppose that $f$ is a time-independent external force such that $\|\La^{-1}f\|_{L^{3}}<c$, where $c<\delta$ is a universal small constant. Let  $U_f\in L^3(\RR^3)$ be the unique solution  to $(SNS)$ with $\|U_f\|_{L^{3}}<2\|\La^{-1}f\|_{L^3}$ (the existence of $U_f$ is provided by Proposition \ref{steady}) .Then we have
	\begin{enumerate}
	    \item For any initial data $u_0\in L^3(\RR^3)$, there exists a unique maximal time $T^*(u_0,f)>0$ and a unique solution  $u_f$ to $(NSf)$ with initial data $u_0$ such that for any $T<T^*(u_0,f)$,
	          \beq
	          u_f\in C([0,T],L^3(\RR^3)).
	          \enq 
	          Moreover there exists a constant $\delta_2(f)$ such that if $\|u_0-U\|_{L^3}<\delta_2$, then $u_0\in C_0(\RR_+,L^{3}(\RR^3))$. The solution $u_f$ satisfies that for $3<p<5$ 
	          \beqq\label{uf-Uf}
	          \lim_{T\to\ T^*(u_0,f)}\|u_f-U_f\|_{\mathbb{L}^{p:\infty}_p(T)}=\infty.
	          \enqq  
	          \item Let $u_f\in C(\RR_+,L^3(\RR^3))$ with initial data $u_0\in L^3(\RR^3)$. Then $u_f\in L^{\infty}(\RR_+,L^3(\RR^3))$ and $u_f-U_f\in\mathbb{L}^{r_0:\infty}_{p}(\infty)$ for some $r_0>2$ and $p>3$, and 	         \beq
	          \lim_{t\to\infty}\|u_f-U\|_{\dot{B}^{s_p}_{p,p}}=0.
	          \enq	     
	\end{enumerate}
\end{theorem}

\subsection{Product laws and heat estimates}
   We first recall the following standard product laws in Besov space, which use the theory of para-products (for details, see \cite{C2,gip2}).  
   \begin{proposition}\label{product-law}
\begin{enumerate}
	\item 

	Let $p>3$, $q>3$ and $r>2$. Moreover assume that $s_q+s_p+\frac{2}{r}>0$.	We have, for any $|\eps|<1$ such that $1-\frac{2}{r}+\eps >0$,
		      \beq
		      \|vw\|_{\mathcal{L}^r([0,T],\dot{B}_{p,p}^{s_p+\frac{2}{r}-1})}\leq C(\eps)\|v\|_{\mathcal{L}^{\infty}([0,T],\dot{B}_{q,q}^{s_q+\eps})}\|w\|_{\mathcal{L}^{r}([0,T],\dot{B}_{p,p}^{s_p+\frac{2}{r}-\eps})}.
		      \enq 
    \item let $p>3$ and $2<r<\frac{2p}{p-3}$. Then for any $\eps\in\RR $ such that $1-\frac{2}{r}-|\eps|>0$, we have
          \beq
          \|vw\|_{\mathcal{L}^{\frac{2}{r}}([0,T],\dot{B}_{p,p}^{s_p+\frac{4}{r}-1})}\leq C(\eps)\|v\|_{\mathcal{L}^{r}([0,T],\dot{B}_{p,p}^{s_p+\frac{2}{r}+\eps })}\|w\|_{\mathcal{L}^{r}([0,T],\dot{B}_{p,p}^{s_p+\frac{2}{r}-\eps })}.
          \enq 
    \item Let $p_1,p_2\in (3,\infty)$, $2<r<\frac{2p}{p-3}$ and $T\in\RR_+\cup\{\infty\}$. Suppose that $v\in\mathbb{L}^{r;\infty}_{p_1}(T)$ and $w\in\mathbb{L}^{r;\infty}_{p_2}(T)$. Then we have 
	\beq
	\|vw\|_{\mathcal{L}^{\frac{r}{2}}([0,T],\dot{B}_{p,p}^{s_p+\frac{4}{r}-1})}\lesssim\|v\|_{\mathbb{L}^{r;\infty}_{p_1}(T)}\|w\|_{\mathbb{L}^{r;\infty}_{p_2}(T)},
	\enq
	where $\frac{1}{p}=\frac{1}{p_1}+\frac{1}{p_2}$.
    \item  Let $p>3$.  Suppose that $w\in L^{\infty}([0,T],L^{3})$ and $v\in\mathcal{L}^{r_0}([0,T];\dot{B}_{p,p}^{s_p+\frac{2}{r_0}})$ for some $T\in\RR_+\cup\{+\infty\}$ with $r_0=\frac{2p}{p-1}$, then we have 
	\beq
	\|vw\|_{\mathcal{L}^{r_0}([0,T],\dot{B}^{s_{\bar{p}}+\frac{2}{r_0}-1}_{\bar{p},\bar{p}})}\leq C(p)\|w\|_{L^{\infty}([0,T],L^{3})}\|v\|_{\mathcal{L}^{r_0}([0,T];\dot{B}_{p,p}^{s_p+\frac{2}{r_0}})},
	\enq 
	where $\frac{1}{\bar{p}}=\frac{1}{3}+\frac{1}{6p}$ and $C(p)\to\infty$ as $p\to\infty$.
    \end{enumerate}
\end{proposition}
Since the first three results in the proposition are standard and well-known, which can be found in \cite{C2,gip2}, we only give the proof of the last of the proposition. 
\begin{proof}
    For simplicity, we treat $w$ and $v$ as functions. We have 
	\beq
	\La_j wv=\La_j T_w v+\La_j T_v w+\La_j R(u,v).
	\enq 
	We first take $q_1$ such that $\frac{1}{\bar{p}}=\frac{1}{p}+\frac{1}{q_1}=\frac{1}{3}+\frac{1}{6p}$ implying that $q_1=\frac{6p}{2p-5}>3$. 
	
		About $\La_j T_w v$, we have
	\beq
	\|\La_j T_w v\|_{L^{r_0}(L^{\bar{p}})}\lesssim\|(S_j w)(\La_j v)\|_{L^{r_0}(L^{\bar{p}})}\lesssim\|S_j w\|_{L^\infty(L^{q_1})}\|v\|_{L^{r_0}(L^p)}.
	\enq 
	And we notice that
	\beq
    \|S_j w\|_{L^\infty(L^{q_1})}\lesssim\sum_{j'\leq j}\|\La_j w\|_{L^\infty(L^{q_1})}\lesssim \sum_{j'\leq j}2^{-j's_{q_1}}c_{j,q_1}\|w\|_{\mathcal{L}^{\infty}([0,T],\dot{B}_{q_1,q_1}^{s_{q_1}})},
	\enq 
	and
	\beq
	\|v\|_{L^{r_0}(L^p)}\lesssim 2^{-j(s_p+\frac{2}{r_0})}c_{j,p}\|v\|_{\mathcal{L}^{r_0}([0,T];\dot{B}_{p,p}^{s_p+\frac{2}{r_0}})}.
	\enq 
	Since $s_{q_1}<0$, we have
	\beq
	\|2^{j(s_{\bar{p}}+\frac{2}{r_0}-1)}\|\La_j T_w v\|_{L^{r_0}(L^{\bar{p}})}\|_{\ell^{\bar{p}}}\lesssim \|w\|_{\mathcal{L}^{\infty}([0,T],\dot{B}_{q_1,q_1}^{s_{q_1}})}\|v\|_{\mathcal{L}^{r_0}([0,T];\dot{B}_{p,p}^{s_p+\frac{2}{r_0}})}.
	\enq  
	This combined with $L^{3}\hookrightarrow \dot{B}^{0}_{3,3}\hookrightarrow\dot{B}^{s_{q_1}}_{q_1,q_1}$, implies that
	\beqq\label{twv}
	\|2^{j(s_{\bar{p}}+\frac{2}{r_0}-1)}\|\La_j T_w v\|_{L^{r_0}(L^{\bar{p}})}\|_{\ell^{\bar{p}}}\lesssim \|w\|_{L^{\infty}([0,T],L^{3})}\|v\|_{\mathcal{L}^{r_0}([0,T];\dot{B}_{p,p}^{s_p+\frac{2}{r_0}})}.
	\enqq 
	
	Now we choose  $q:=\frac{12p}{4p-1}$ and $p_1:=4p$. It is easy to check such that $\frac{1}{\bar{p}}=\frac{1}{p_1}+\frac{1}{q}=\frac{1}{3}+\frac{1}{6p}$ . We notice that 
	\beq
	\|\La_j T_v w\|_{L^{r_0}(L^{\bar{p}})}\lesssim\|S_j v\|_{L^{r_0}(L^{p_1})}\|\La_j w\|_{L^{\infty}(L^q)},
	\enq 
	and
	\beq
	\|S_j v\|_{L^{r_0}(L^{p_1})}&\lesssim&\sum_{j'\leq j}\|\La_j v\|_{L^{r_0}(L^{p_1})}\lesssim 2^{-j'(s_{p_1}+\frac{2}{r_0})}c_{j',p}\|v\|_{\mathcal{L}^{r_0}([0,T];\dot{B}_{p_1,p}^{s_{p_1}+\frac{2}{r_0}})}\\
	&\lesssim&2^{-j'(s_{p_1}+\frac{2}{r_0})}c_{j',p}\|v\|_{\mathcal{L}^{r_0}([0,T];\dot{B}_{p,p}^{s_{p}+\frac{2}{r_0}})},
	\enq 
	and 
	\beq
	\|\La_j w\|_{L^{\infty}(L^q)}\lesssim 2^{-js_q}c_{j,q}\|w\|_{\mathcal{L}^{\infty}([0,T],\dot{B}_{q,q}^{s_{q}})}.
	\enq 
	Since  $s_{p_1}+\frac{2}{r_0}=-1+\frac{3}{4p}+1-\frac{1}{p}=-\frac{1}{4p} <0$,  we have 
	\beqq\label{tvw}
	\|2^{j(s_{\bar{p}}+\frac{2}{r_0}-1)}\|\La_j T_w v\|_{L^{r_0}(L^{\bar{p}})}\|_{\ell^{\bar{p}}}\lesssim p \|w\|_{\mathcal{L}^{\infty}([0,T],\dot{B}_{q,q}^{s_{q}})}\|v\|_{\mathcal{L}^{r_0}([0,T];\dot{B}_{p,p}^{s_{p}+\frac{2}{r_0}})}.
	\enqq 
	Again by Lemma  $L^{3}\hookrightarrow \dot{B}^{0}_{3,3}\hookrightarrow\dot{B}^{s_{q}}_{q,q}$, we have
	\beqq\label{tvw}
	\|2^{j(s_{\bar{p}}+\frac{2}{r_0}-1)}\|\La_j T_w v\|_{L^{r_0}(L^{\bar{p}})}\|_{\ell^{\bar{p}}}\lesssim p \|w\|_{L^{\infty}([0,T],L^{3})}\|v\|_{\mathcal{L}^{r_0}([0,T];\dot{B}_{p,p}^{s_{p}+\frac{2}{r_0}})}.
	\enqq 
	Now we turn to the remainder $\La_j R(w,v)$. We denote that $\frac{1}{\tilde{p}}:=\frac{1}{p}+\frac{1}{q}=\frac{1}{3}+\frac{11}{12p}$.
	Since 
	\beq
	\begin{aligned}
	\|\La_j R(w,v)&\|_{L^{r_0}(L^{\tilde{p}})}\lesssim\sum_{k\geq j-1}\|\La_k w\|_{L^{\infty}(L^q)}\|\La_k v\|_{L^{r_0}(L^p)}\\
	&\lesssim\sum_{k\geq j}2^{-k(s_q+s_p+\frac{2}{r_0})}c_{k,q}c_{k,p}\|w\|_{\mathcal{L}^{\infty}([0,T],\dot{B}_{q,q}^{s_{q}})}\|v\|_{\mathcal{L}^{r_0}([0,T];\dot{B}_{p,p}^{s_{p}+\frac{2}{r_0}})},
	\end{aligned}
	\enq 
	and  
	\beq
	s_p+s_q+\frac{2}{r_0}=\frac{7}{4p}>0
	\enq 
	 we have that, by applying $L^{3}\hookrightarrow \dot{B}^{0}_{3,3}\hookrightarrow\dot{B}^{s_{q}}_{q,q}$,
	\beq
	\begin{aligned}
	\|2^{j(s_{\tilde{p}}+\frac{2}{r_0}-1)}\|\La_j R(w,v)\|_{L^{r_0}(L^{\tilde{p}})}\|_{\ell^{\tilde{p}}}\lesssim&p\|w\|_{\mathcal{L}^{\infty}([0,T],\dot{B}_{q,q}^{s_{q}})}\|v\|_{\mathcal{L}^{r_0}([0,T];\dot{B}_{p,p}^{s_{p}+\frac{2}{r_0}})}\\
	\lesssim &p\|w\|_{L^{\infty}([0,T],L^{3})}\|v\|_{\mathcal{L}^{r_0}([0,T];\dot{B}_{p,p}^{s_{p}+\frac{2}{r_0}})}
	\end{aligned}
	\enq 

	which is $R(w,v)\in\mathcal{L}^{r_0}([0,T];\dot{B}_{\tilde{p},\tilde{p}}^{s_{\tilde{p}}+\frac{2}{r_0}-1})$.\\
	And we have $R(w,v)\in\mathcal{L}^{r_0}([0,T];\dot{B}_{\bar{p},\bar{p}}^{s_{\bar{p}}+\frac{2}{r_0}-1})$, as $\tilde{p}<\bar{p}$. Combining with \reff{twv} and \reff{tvw} we get
	\beq
	\|w v\|_{\mathcal{L}^{r_0}([0,T];\dot{B}_{\bar{p},\bar{p}}^{s_{\bar{p}}+\frac{2}{r_0}-1})}\leq C(p) \|w\|_{L^{\infty}([0,T],L^{3})}\|v\|_{\mathcal{L}^{r_0}([0,T];\dot{B}_{p,p}^{s_{p}+\frac{2}{r_0}})},
	\enq 
	where $C(p)\to\infty$ as $p\to\infty$.
	The proposition is proved.
\end{proof}
Now let us recall the following standard heat estimate. For any $p\in[1,\infty]$, there exists some $c_0,c>0$ such that for any $f\in\mathcal{S}'$ and $j\in\mathbb{Z}$,
\beq
\|\La_{j}(e^{t\La }f)\|_{L^p}\leq c_0e^{-ct2^{2j}}\|\La_j f\|_{L^{p}}.
\enq 
Hence for $0<t\leq \infty$, recalling 
\beq
B(u,v):=\int_0^t e^{(t-s)\La }\mathbb{P}\nabla\cdot(u(s)\otimes v(s))ds,
\enq 
Young's inequality for convolutions implies that for any $\tilde{r}\in[r,\infty]$
\beqq
\|B(u,v)\|_{\mathcal{L}^{\tilde{r}}([0,T],\dot{B}_{p,p}^{s+2+2(\frac{1}{\tilde{r}}-\frac{1}{r})})}\lesssim\|u\otimes v\|_{\mathcal{L}^r([0,T],\dot{B}_{p,p}^{s+1})}.
\enqq

And we recall that $B$ is a bounded operator from $L^{\infty}([0,T],L^{3,\infty})\times L^{\infty}([0,T],L^{3,\infty})$ to $L^{\infty}([0,T],L^{3,\infty})$ for any $T\in \RR_+\cup\{+\infty\}$ (see \cite{bbis2})

\subsection*{Acknowledgement}
The author is grateful to Isabelle Gallagher for sharing many hindsights about the Navier-Stokes equation and entertaining discussions for overcoming the difficulties during this research.



\begin{thebibliography}{99}
\bibitem{DA2} D. Albritton, Blow-up criteria for the Navier-Stokes equations in non-endpoint critical Besov spaces. arXiv: 1612.04439
\bibitem{bbis2}  C. Bjorland, L. Brandolese, D. Iftimie \& M. E. Schonbek, $L^p$-solutions of the steady-state Navier-Stokes equations with rough external forces. {\it Communications in Partial Differential Equations}, 2010, 36(2), 216-246

\bibitem{BP2} J. Bourgain \& N.  Pavlović, Ill-posedness of the Navier-Stokes equations in a critical space in 3D
{\it J. Funct. Anal.} 255 (2008), no. 9, 2233–2247. 

\bibitem{CM2} M. Cannone, A generalization of a theorem by Kato on Navier-Stokes equations. {\it Rev. Mat. Iberoamericana} 13 (1997), no. 3, 515–-541. 
\bibitem{CMP2} M. Cannone, Y. Meyer \& F. Planchon, Solutions auto-similaires des équations de Navier-Stokes. 
{\it Séminaire sur les Équations aux Dérivées Partielles, 1993–-1994,}  No. VIII, 12 pp., École Polytech., Palaiseau, 1994. 

\bibitem{CP2} M. Cannone \& F. Planchon, On the non-stationary Navier-Stokes equations with an external force. {\it Adv. Differential Equations} 4 (1999), no. 5, 697–730.

\bibitem{C2} J.-Y Chemin, Théorèmes d'unicité pour le système de Navier-Stokes tridimensionnel. (French) [Uniqueness theorems for the three-dimensional Navier-Stokes system]. {\it  J. Anal. Math.}77 (1999), 27–50. 

\bibitem{CG}J.-Y Chemin \& I. Gallagher, Wellposedness and stability results for the Navier-Stokes equations in $\mathbb{R}^3$. {\it Ann. Inst. H. Poincaré Anal. Non Linéaire}, 26 (2009), no. 2, 599–624. 

\bibitem{CL2} J.-Y Chemin \& N. Lerner, Flot de champs de vecteurs non lipschitziens et équations de Navier-Stokes. {\it J. Differential Equations} 121 (1995), no. 2, 314–328.


\bibitem{ESS12} L. Escauriaza, G. Seregin \& V. Šverák, $L_{3,\infty}$--solutions of Navier-Stokes equations and backward uniqueness. {\it  Uspekhi Mat. Nauk.} 58 (2003), no. 2(350), 3--44.
\bibitem{ESS22} L. Escauriaza, G. Seregin \& V. Šverák, Backward uniqueness for the heat operator in half-space. {\it Algebra i Analiz.} 15 (2003), no. 1, 201--214.
\bibitem{fk2} H. Fujita \& T. Kato, On the Navier-Stokes initial value problem. I. 
{\it Arch. Rational Mech. Anal.} 16 1964 269–-315.

\bibitem{IG12} I. Gallagher, Profile decomposition for solutions of the Navier-Stokes equations. {\it Bull. Soc. Math. France}
129 (2001), no. 2, 285–316.

\bibitem{gip2} I. Gallagher, D. Iftimie \& F. Planchon, Asymptotics and stability for global solutions to the Navier-Stokes equations.
{\it Ann. Inst. Fourier (Grenoble)} 53 (2003), no. 5, 1387–-1424. 
\bibitem{gkp2} I. Gallagher, G. S. Koch \& F. Planchon, Blow-up of critical Besov norms at a potential Navier-Stokes singularity. {\it  Comm. Math. Phys.} 343 (2016), no. 1, 39–82. 
\bibitem{gkp12} I. Gallagher, G. S. Koch \& F. Planchon, A profile decomposition approach to the $L^{\infty}_t(L^3_x)$ Navier-Stokes regularity criterion. {\it Math. Ann.}355 (2013), no. 4, 1527–1559. 
\bibitem{PG2} P. Gérard, Description du défaut de compacité de l'injection de Sobolev. {\it ESAIM Control Optim. Calc. Var.}3 (1998), 213–233. 


\bibitem{Kato2} T. Kato, Strong $L^p$ solutions of the Navier-Stokes equation in $\RR^m$, with applications to weak solutions.
{\it Mathematische Zeitschrift}, 187, 1984, pages 471--480.
\bibitem{KK2} C. E. Kenig \& G. S. Koch, An alternative approach to regularity for the Navier-Stokes equations in critical spaces. {\it Ann. Inst. H. Poincaré Anal. Non Linéaire}28 (2011), no. 2, 159–187.












\bibitem{K12}G. S. Koch, Profile decompositions for critical Lebesgue and Besov space embeddings. {\it Indiana Univ. Math. J. } 59 (2010), no. 5, 1801–1830. 
\bibitem{KT2} H. Koch \& D. Tataru, Well-posedness for the Navier-Stokes equations. {\it Adv. Math.} 157 (2001), no. 1, 22–35.
\bibitem{L2} J. Leray, Sur le mouvement d'un liquide visqueux emplissant l'espace. {\it Acta Math.} 63 (1934), no. 1, 193–248.
\bibitem{NRS2} J. Nečas, M.  Růžička \& V.  Šverák, On Leray's self-similar solutions of the Navier-Stokes equations.
{\it  Acta Math.} 176 (1996), no. 2, 283–294.
\bibitem{FP2}F. Planchon, Asymptotic behavior of global solutions to the Navier-Stokes equations in $\RR^3$. {\it  Rev. Mat. Iberoamericana} 14 (1998), no. 1, 71–93.
\bibitem{GS2} G. Seregin, A certain necessary condition of potential blow up for Navier-Stokes equations. {\it  Comm. Math. Phys.} 312 (2012), no. 3, 833–845.

\bibitem{DW2} D. Wu, Cauchy problem for the incompressible Navier-Stokes equation with an external force. arXiv:1712.05211 











\end{thebibliography}
\end{document}